\documentclass{article}
\usepackage{amsthm,amssymb}
\usepackage{mathrsfs}
\usepackage{yfonts}
\usepackage[margin=1in]{geometry}
\usepackage{amsmath,amscd}
\usepackage{enumerate}
\usepackage{wasysym}
\usepackage{url}

\usepackage{ytableau}
\usepackage{multicol}

\usepackage{color}
\usepackage{float}

\newcommand{\comments}[1]{}

\newcommand\T{\rule{0pt}{2.6ex}}
\newcommand\B{\rule[-1.2ex]{0pt}{0pt}}

\newcommand{\C}{\mathbb{C}}

\newcommand{\Z}{\mathbb{Z}}

\newcommand{\sN}{\mathscr{N}}
\newcommand{\sM}{\mathscr{M}}

\newcommand{\sO}{\mathscr{O}}
\newcommand{\sA}{\mathscr{A}}
\newcommand{\sNt}{\widetilde{\mathscr{N}}}
\newcommand{\sMt}{\widetilde{\mathscr{M}}}
\newcommand{\Bx}{\mathscr{B}_x}
\newcommand{\Ox}{\mathscr{O}_x}

\newcommand{\Mx}{\widetilde{\mathscr{M}}_x}
\newcommand{\w}{\sigma}
\newcommand{\W}{\mathscr{W}}

\newcommand{\Qlb}{\overline{\mathbb{Q}}_{\ell}}

\newcommand{\hB}{B}

\newcommand{\kg}{\mathfrak{g}}

\newcommand{\ku}{\mathfrak{u}}

\newcommand{\ksl}{\mathfrak{sl}}
\newcommand{\row}{\mathrm{row}}
\newcommand{\col}{\mathrm{col}}
\newcommand{\height}{\mathrm{ht}}
\newcommand{\kl}{\mathfrak{l}}

\newcommand{\IC}{\mathrm{IC}}

\DeclareMathOperator{\Spec}{Spec}
\DeclareMathOperator{\multi}{mult}
\DeclareMathOperator{\Hom}{Hom}
\DeclareMathOperator{\Lie}{Lie}

\DeclareMathOperator{\Ad}{Ad}

\newtheorem{thm}{Theorem}[section]
\newtheorem{cor}[thm]{Corollary}
\newtheorem{lem}[thm]{Lemma}
\newtheorem{prop}[thm]{Proposition}
\newtheorem{fact}[thm]{Fact}

\theoremstyle{definition}
\newtheorem{defn}[thm]{Definition}
\newtheorem{example}[thm]{Example}

\newenvironment{remark}[1][Remark.]{\begin{trivlist}
\item[\hskip \labelsep \textit{#1}]}{\end{trivlist}}

\begin{document}

\title{Graham's Variety and Perverse Sheaves on the Nilpotent Cone}
\author{Amber Russell}
\date{\today}
\maketitle

\begin{abstract}
Graham has constructed a variety with a map to the nilpotent cone which is similar in some ways to the Springer resolution.  One aspect in which Graham's map differs is that it is not in general an isomorphism over the principal orbit, but rather the universal covering map.  This map gives rise to a certain semisimple perverse sheaf on the nilpotent cone, and we discuss here the problem of describing the summands of this perverse sheaf.  For type $A_n$, a key tool is a known description of an affine paving of Springer fibers.
\end{abstract}

%

\section{Introduction}

Let $\kg$ be a complex semisimple Lie algebra with simply-connected algebraic group $G$ and Weyl group $\mathscr{W}$.  Let $Z$ be the center of $G$ and $G_{ad} = G/Z$ the adjoint group associated to $\kg$.  Fix a Borel subgroup $B$ in $G$ and let $\ku$ be the unipotent radical of $\Lie B$.  We can then define $\sNt = G\times^B \ku$ and the map $\mu:\sNt \rightarrow \kg$ which sends $(g, x)$ to $g\cdot x$.  The image of $\mu$ is the nilpotent cone $\sN$ in $\kg$, and $\mu$ is called the Springer resolution.  For any $x\in \sN$, let $\sO_x$ be the nilpotent orbit containing $x$ under the action of $G$.

In \cite{Springer1976} and \cite{Springer1978}, Springer first associated every irreducible representation of $\mathscr{W}$ to a nilpotent orbit $\sO_x$ and a local system on that orbit.  This correspondence, known as the Springer correspondence, is not generally a bijection between nilpotent orbits and Weyl group representations, nor between Weyl group representations and local systems on the orbits.  Every orbit appears at least once.  However, not all local systems appear, and all that do are $G_{ad}$-equivariant.

Many proofs of Springer's classical results are now known.  In particular, Borho and MacPherson prove it in \cite{Borho1983} by using the Decomposition Theorem for perverse sheaves to study $\mathrm{R}\mu_*\underline{\Qlb}_{\sNt}$ (where $\underline{\Qlb}_{\sNt}$ denotes the constant sheaf).  The goal of this paper is to begin an analogous construction using a variety described by Graham in \cite{Graham}.  Here, $\sNt$ and the Springer resolution are replaced by Graham's variety $\sMt$ and a new map $\tilde{\mu}:\sMt \rightarrow \sN$ which is no longer an isomorphism over the principal orbit of $\sN$.  The map $\tilde{\mu}$ factors through the Springer resolution and we see that when the Decomposition Theorem is applied, the local systems from the Springer correspondence are recovered, and additional ones occur as well in the cases where $G_{ad} \neq G$.  The main result of this paper is that in type $A_n$, where only the trivial local systems appear in the Springer correspondence, all of the $G$-equivariant local systems occur as a result of Graham's variety.   Even though the main focus will be on type $A_n$, several intermediate steps are valid in other types and are thus proven more generally.

The study of the local systems missing from the Springer correspondence is not new, and a rather different construction by Lusztig also addresses this issue.  His techniques have led to the generalized Springer correspondence and character sheaves.  (See \cite{Lu84}.)  In new collaborative work with William Graham and Martha Precup, we have been exploring the connection between this $\sMt$ and Lusztig's work. 

Section 2 reviews Graham's construction of $\sMt$ and $\tilde{\mu}$, and Section 3 explores the relationship between the torus orbits of Graham's construction and the nilpotent orbits of $G$.  In Section 4, the fibers of $\tilde{\mu}$ in type $A_n$ are studied using Tymoczko's combinatorial description of an affine paving of Springer fibers \cite{TymLin}.  The main results involving perverse sheaves can be found in Section 5.  An appendix contains tables that treat the results of Section 3 in the applicable exceptional types.

\section{Graham's Variety}

Let $\sO_{prin}$ be the principal orbit in $\sN$ and $\sO$ be its universal cover. Since $\sN$ is a normal variety \cite{Kostant1963}, $\sN = \Spec \mathrm{R}(\sO_{prin})$ where $\mathrm{R}(\sO_{prin})$ is the ring of regular functions for the principal orbit.  Let $\sM = \Spec \mathrm{R}(\sO)$.  In \cite{Graham}, Graham defines a map $\sMt \rightarrow \sM$ which is in some ways analogous to the Springer resolution for $\sM$.  He also gives a map $\tilde{\mu}: \sMt \rightarrow \sN$ which factors through the Springer resolution $\mu:\sNt\rightarrow \sN$.  It is this map $\tilde{\mu}$ and its relationship to the Springer resolution which is the focus of this paper.

Let $B$ be a Borel subgroup of $G$, and let $T$ be a maximal torus of $B$.  We know that the center $Z$ of $G$ is contained in $B$, so $B_{ad} = B/Z$ is a Borel subgroup of $G_{ad}$ with maximal torus $T_{ad}=T/Z$.   With this choice of a Borel subgroup and a maximal torus, we will assume the positive roots to be those in the Borel, so that the simple roots are determined by this choice.  Let $\alpha_1,...,\alpha_n$ denote these simple roots following the conventions of \cite{HumphreysLieAlg}.  Define $W_{ad}\subset \kg$ to be the span of the simple root vectors.  Note that $W_{ad}\cong \ku/[\ku,\ku]$.  Then $W_{ad}$ is a toric variety for $T_{ad}$ with lattice given by the character group of $T_{ad}$ and cone generated by the fundamental weights.  By following the same toric variety construction as for $W_{ad}$ but with the character group for $T$, we can construct a toric variety $W$ for $T$ such that $W/Z = W_{ad}$. This variety $W$ is not in general smooth.  However, in the case where $G=G_{ad}$, we have $W=W_{ad}$ and $\sMt = \sNt$. For a more detailed explanation of the construction of a toric variety, see \cite{Fulton} or \cite{Graham}.  

The Borel $B$ acts on $W$ through projection onto $T$, so that the unipotent part of $B$ acts trivially.  The action of $B$ on $W_{ad}$ comes from viewing $W_{ad}$ as the quotient of the nilradical $\ku/[\ku,\ku]$.  Let $p:\ku \rightarrow W_{ad} (= \ku/[\ku,\ku])$ and $q:W\rightarrow W_{ad}( = W/Z)$ be the $B$-equivariant projections.  We can then define $$\tilde{\ku} = W \times_{W_{ad}} \ku = \{(w,u)\ |\ p(u)=q(w)\}.$$  To clarify this, consider the diagram

\begin{center}
$\begin{CD}
\widetilde{\ku} @>\overline{q}>> \ku\\
@V\overline{p} VV @VVp V\\
W @>q >> W_{ad}
\end{CD}$
\end{center} where $\overline{p}$ and $\overline{q}$ complete our cartesian product.

\begin{defn}
\textit{Graham's variety}, $\sMt$, is $G \times^B \tilde{\ku}$. 
\end{defn}

In his paper, Graham outlines a map $\widetilde{\ku}\rightarrow \sM$ which he then extends to his analog of the Springer resolution $\sMt\rightarrow \sM$.  This map is proper and an isomorphism over the universal cover $\sO$ of the principal orbit.  However, it is not a resolution of singularities since $\sMt$ is not generally smooth.  

Of more interest to this paper, the quotient map $\overline{q}:\widetilde{\ku}\rightarrow\ku$ induces a finite quotient map $\gamma: G\times^B\widetilde{\ku}\rightarrow G\times^B\ku$.  Thus, we have a diagram \[\sMt\stackrel{\gamma}{\longrightarrow}\sNt\stackrel{\mu}{\longrightarrow}\sN\] 
where $\mu$ is our familiar Springer resolution.  These combine to give the result below.

\begin{thm}[Graham \cite{Graham}]
Let $\sMt$ and $\gamma$ be as above.  Then, the following diagram commutes.

\begin{center}
$\begin{CD}
\sMt @>\gamma>> \sNt\\
@V VV @VV\mu V\\
\sM @>>\ > \sN
\end{CD}$
\end{center}

\end{thm}


Note, the statement of the above theorem is not needed here, but it is included to help describe Graham's work. We will now restrict our focus to the map $\tilde{\mu} :=  \mu \circ \gamma$.  The map $\mu$ is the familiar Springer resolution, so we will start by attempting to understand $\gamma$.


The map $\gamma:\sMt\rightarrow \sNt$ ultimately comes from the quotient map $q : W \rightarrow W_{ad}$, and Graham gives a method to determine the fibers of $q$ in terms of the $T_{ad}$-orbits in $W_{ad}$.  Fix $\alpha_1,...,\alpha_n$ to be simple roots following the conventions of \cite{HumphreysLieAlg}.  Let $X_\alpha\in \kg$ be the root vector corresponding to the root $\alpha$.  Then, $W_{ad}$ is isomorphic to the span of $\{X_{\alpha_i}\ |\ \alpha_i$ is a simple root\}.  The $T_{ad}$-orbits correspond to faces of the cone of $W_{ad}$, and we can describe them using subsets $J$ of $\{1,2,\ldots ,n\}$.  Each subset $J$ corresponds to the simple roots that are not included in the face.  We will denote by $\tau_J$ the face of the cone and by $\sO_{ad}(\tau_J)$ the $T_{ad}$-orbit corresponding to the subset $J$, so $\sO_{ad}(\tau_J)  = \mathrm{Span}\{ X_{\alpha_i}\ |\ i\notin J\}$.  We will use $Z(J)$ to denote the fiber of $q$ over any element in that orbit.  In fact, $Z(J)$ has a group structure and can be seen as a subgroup of $T$.  We will see this in the upcoming proof of Proposition 3.2.  Note that by the definition of $\sMt$, the fibers of $q$ determine the fibers of $\gamma$.

\begin{thm}\label{ZJ}
Let $J$ be a nonempty subset of $\{1, 2, \ldots , n\}$.  Then, the group $Z(J)$ is given below for the simple $\mathfrak g$ in each relevant type.
\end{thm}

\begin{table}[H]
\begin{flushleft}
\begin{tabular}{l l l l}
\B $A_n$ : & $Z(J)=$ & $\Z/c\Z$ & \hspace{5mm} where $c = \gcd(J\cup\{n+1\})$\\
 \T \B $B_n$ : & $Z(J)=$ & $\Z/2\Z$ & \hspace{5mm} if all $j \in J$ are even\\
\B & $Z(J)=$ & $\{1\}$ & \hspace{5mm} otherwise\\
\T \B $C_n$ : & $Z(J)=$ & $\Z/2\Z$ & \hspace{5mm} if $n \notin J$\\
\B & $Z(J)=$ & $\{1\}$ & \hspace{5mm} otherwise\\
%
\T \B $D_n$ : & $Z(J)=$ & $\Z/2\Z\times\Z/2\Z$ & \hspace{5mm} if $n-1,n \notin J$ and all $j\in J$ are even\\
\B & $Z(J)=$ & $\Z/2\Z$ & \hspace{5mm} if $n-1,n \notin J$ and not all $j\in J$ are even\\
 & $Z(J)=$ & $\Z/2\Z$ & \hspace{5mm} if exactly one of $n-1$ and $n$ is in $J$,\\
 & & & \hspace{5mm} all $j\in J$ such that $j<n-1$ are even,\\
 \B & & & \hspace{5mm} and $n = 4k+2$ for some $k\geq 1$\\
  \B & $Z(J)=$ & $\{1\}$ & \hspace{5mm} otherwise\\
\T \B $E_6$ : & $Z(J)=$ & $\Z/3\Z$ & \hspace{5mm} if none of $1,3,5,6$ are in $J$;\\
\B & & &  \hspace{5mm} else $Z(J) = \{1\}$\\
\T \B $E_7$ : & $Z(J)=$ & $\Z/2\Z$ & \hspace{5mm} if none of $2,5,7$ are in $J$;\\
\B & & &  \hspace{5mm} else $Z(J) = \{1\}$\\
\end{tabular}
\end{flushleft}
\end{table}

\begin{proof}
As noted above, an orbit in the toric variety corresponds to a face of the cone.  Then, an orbit $\sO(\tau_J)$ is isomorphic to the torus $T(\tau_J) = \Hom(\tau_J^{\perp}\cap \widehat{T},\C^*)$ where $\widehat{T}$ is the character group for $T$, i.e. the weight lattice.  Analogous statements and definitions for the torus $T_{ad}$ can be made.  Let us define \[Z(J) = \ker(T(\tau_J)\rightarrow T_{ad}(\tau_J)).\] Then the character group for $Z(J)$ is \[\widehat{Z}(J) = \widehat{T}(\tau_J)/ \widehat{T}_{ad}(\tau_J) = (\tau_J^{\perp} \cap\widehat{T})/(\tau_J^{\perp}\cap \widehat{T}_{ad}).\]  Graham proves in \cite{Graham} that $\sO(\tau_J) \rightarrow \sO_{ad}(\tau_J)$ is a covering map with fibers $Z(J)$.  To determine $\widehat{Z}(J)$, note that $(\tau_J^{\perp} \cap\widehat{T})/(\tau_J^{\perp}\cap \widehat{T}_{ad})$ maps injectively into $\widehat{T}/\widehat{T}_{ad} =Z$ which is the abstract fundamental group of the root system.  To see the elements in $Z$, we will follow the conventions found in Section 13.2 of \cite{HumphreysLieAlg}, where a table can be found that lists the dominant weights.  The elements in this quotient group can be represented by particular dominant weights $\lambda_i$, where two dominant weights are in the same coset if their difference can be written with integer coefficients for all simple roots $\alpha_j$.  To describe the image of $\widehat{Z}(J)$ in $Z$, and thus determine $Z(J)$, a coset $\lambda + \widehat{T}_{ad}$, where $\lambda$ is a dominant weight, is in the image if it has nonempty intersection with $\tau_J^{\perp}$, which means the coefficients for all $\alpha_j$ in $\lambda$ are integers if $j\in J$.   Directly checking when these coefficients are integers gives the table above.  For more details, see \cite{diss}.

\end{proof}

\section{$G$-orbits and Graham's Fibers}
The Bala--Carter theorem associates to each orbit $\Ox\subset \sN$ a pair $(\kl, \sO_{x,\kl})$, where $\kl$ is the smallest Levi subalgebra meeting $\Ox$ and $\sO_{x,\kl} = \Ox \cap \kl$.  In these terms, the orbits $\Ox$ that meet $W_{ad}$ are those for which $\sO_{x,\kl}$ is principal in $\kl$.  In other words, those orbits $\Ox$ so that some element in $\Ox$ is a sum of simple root vectors. In type $A_n$, all of the $G$-orbits are of this type, so they all intersect $W_{ad}$.  In the other types, not all orbits have this property.  In the following theorem, we give the inclusion correspondence between $T_{ad}$-orbits in $W_{ad}$ and $G$-orbits in $\sN$ for the classical types in terms of the $G$-orbits'  partition classification.  Note that multiple $T_{ad}$-orbits can be contained in the same $G$-orbit, and we are viewing $W_{ad}$ as a subset of $\sN$ in this instance rather than the quotient $\ku/[\ku,\ku]$. The partition notation comes from \cite{NilOrbBk}, except that there may be repetition of some parts and the order may not be decreasing.  For example, $[3^2\ 1]$ may appear as $[3\ 3\ 1]$, $[3\ 1\ 3]$, or $[1\ 3\ 3]$.  For this reason, we will refer to the partitions in the following statement as \textit{unreduced}.  This correspondence for the exceptional types can be found in the Appendix.

\begin{prop}\label{part}
Let $\kg$ be a Lie algebra of classical type.  Let $J = \{ d_1, d_2,\ldots , d_r\}$ be a subset of $\{1, 2,\ldots , n\}$ with the assumption that $d_i < d_j$ if $i<j$ .  Then $\sO_{ad}(\tau_J)$ is contained in the $G$-orbit given by the unreduced partition ${P_J}$.  

\begin{table}[h]
\begin{flushleft}
\begin{tabular}{l@{}l@{}l}
$A_n$ : & $\ {P_J}=$ &  $[(n+1-d_r)\ (d_r - d_{r-1})\ \ldots \ (d_2 - d_1)\ d_1]$ \\\\
$B_n$ : & $\ {P_J}=$ & $[\ (2(n-d_r)+1)\ (d_r - d_{r-1})^2\ldots \ (d_2 - d_1)^2\ d_1^2\ ]$ \\\\
$C_n$ : & $\ {P_J}=$ & $[\ 2(n-d_r)\ (d_r - d_{r-1})^2\ldots \ (d_2 - d_1)^2\ d_1^2\ ]$ \\\\
$D_n$ : & $\ {P_J}=$ & $[\ (2(n-d_r)-1)\ (d_r - d_{r-1})^2\ldots \ (d_2 - d_1)^2\ d_1^2\ 1\ ] $ if $n-1, n \notin J$\\
\T & $\ {P_J}=$ & $[\ (d_r - d_{r-1})^2\ldots \ (d_2 - d_1)^2\ d_1^2\ ] $ if $n-1, n \in J$\\
\T & $\ {P_J}=$ & $[\ (n-d_{r-1})^2\ (d_{r-1} - d_{r-2})^2\ldots \ (d_2 - d_1)^2\ d_1^2\ ] $ if $n-1$ or $n\in J$
\end{tabular}
\end{flushleft}
\end{table}

\end{prop}

\begin{proof}
We take a representative $X_J$ of a set $J$ to be the sum of the root vectors $X_{\alpha_i}$ for all $i \in \{1, 2,\ldots , n\}-J$ where $\alpha_i$ is a simple root following the notation of Humphreys \cite{HumphreysLieAlg}.  Since each $T_{ad}$-orbit is contained in some $G$-orbit, by calculating the Jordan canonical form for our representative, we are able to associate a single $G$-orbit to each $J$.  We will use the root vector conventions found in \cite{NilOrbBk}.  

In type $A_n$, the root vector $X_{\alpha_i} = E_{i,\ i+1}$ where $E_{i,j}$ is a matrix with a one in the $i$th row and $j$th column and zeroes everywhere else.  In this case, we see $X_J$ is already in Jordan canonical form, and consecutive numbers not in $J$ give us the sizes of the Jordan blocks.  Thus, the formula for ${P_J}$ is given by the distance between the consecutive elements in $J$, the distance between the largest element in $J$ and $n+1$, and the distance between the smallest element and zero. 

Suppose we are in type $B_n$.  Then, $X_{\alpha_i} = E_{i+1,\ i+2} - E_{n+i+2,\ n+i+1}$ if $1\leq i\leq n-1$ and $X_{\alpha_n} = E_{1,\ 2n+1} - E_{n+1,\ 1}$.  When putting $X_J$ in Jordan canonical form, we see that $E_{i,j}$ and $E_{j,k}$ give rise to elements in the same block, and this determines the blocks, so long as each $i$ occurs only once as the first index and once as the second.  From the definitions of $X_\alpha$, we see that a maximal set of consecutive roots creates two blocks, each with size equal to the size of the set, except when one of the roots is $\alpha_n$.  The root vector $X_{\alpha_n}$ forms a single block of size three when considered alone. Consequently, if $\alpha_n$ is included in the set, the block formed has size one more than twice the size of the set.  As before, we compute the sizes of the blocks by taking the distance between consecutive elements not in $J$. However, each distance now corresponds to two blocks instead of one, unless $J$ does not contain $n$.

Now, suppose we are in type $C_n$.  In this case, $X_{\alpha_i} = E_{i,\ i+1} - E_{n+i+1,\ n+i}$ if $1\leq i\leq n-1$ and $X_{\alpha_n} = E_{n,\ 2n}$. As in type $B_n$, a maximal set of consecutive roots creates two blocks with size equal to the size of the set with the exception of when $\alpha_n$ is in the set.  Any maximal set of $k$ consecutive roots containing $\alpha_n$ forms a block of size $2k$.  So again, we compute the sizes of the blocks by taking the distance between consecutive elements not in $J$ with each distance corresponding to two blocks except when $n$ is not contained in $J$.  

Finally, suppose we are in type $D_n$.  Then, $X_{\alpha_i} = E_{i,\ i+1} - E_{n+i+1,\ n+i}$ if $1\leq i\leq n-1$ and $X_{\alpha_n} = E_{n-1,\ 2n} - E_{n,\ 2n-1}$. We now have three distinct cases to consider. To begin, suppose both $n-1$ and $n$ are in $J$.  Then, we can proceed as in type $C_n$.  Next, suppose neither $n-1$ nor $n$ is in $J$.  This works as before except $X_{\alpha_{n-1}} + X_{\alpha_{n}}$ forms one block of size three and another of size one.  Then, each consecutive root for $\alpha_{n-1}$ adds two to the size of the larger block.  This gives the block of size $2(n-d_r)-1$ appearing in the formula.  Lastly, suppose exactly one of $n$ and $n-1$ is in $J$. The block calculation follows the same in either case due to the definitions of $X_{\alpha_{n-1}}$ and $X_{\alpha_n}$, so that only $n$ appears in the formula.  

\end{proof}

\begin{remark}
A partition is called \textit{very even} if it has only even parts, each with even multiplicity.  In the case of type $D_n$, very even partitions give two distinct nilpotent orbits.  In order to completely classify the correspondence between $T_{ad}$-orbits and $G$-orbits, the weighted Dynkin diagrams would need to be calculated if ${P_J}$ is very even in type $D_n$. This can only happen if ${P_J}$ is of the third type listed in Proposition \ref{part}, and the result should be that the orbit is determined by whether $n$ or $n-1$ is in $J$.  Since no distinction between these two orbits is necessary for the purposes of this paper, the calculation is omitted.  See \cite{NilOrbBk} for the weighted Dynkin diagrams of the two orbits.
\end{remark}

For any algebraic group $G$, there is a $G$-equivariant fundamental group for each nilpotent orbit $\Ox$ defined to be the component group of the centralizer of $x$ in $G$.  That is, $G^{x}/(G^x)^\circ$.  If $G$ is the simply-connected algebraic group, this $G$-equivariant fundamental group is isomorphic to the actual fundamental group of the orbit, so we will denote this component group by $\pi_1(\Ox)$.  If $G$ is the adjoint group $G_{ad}$, we denote the $G_{ad}$-equivariant fundamental group by $A(\Ox)$.  These groups are known for each orbit in every type.  Note that the map $G\rightarrow G/Z = G_{ad}$ induces a map from $\pi_1(\Ox)$ to $A(\Ox)$.  

Let $P$ be a partition of $n$ of the appropriate form for each classical Lie algebra, and let $\sO_x$ be the nilpotent orbit associated to $P$.  In Table \ref{tab:fund}, we give formulas for the fundamental group $\pi_1(\mathscr{O}_x)$ of each orbit and the $G_{ad}$-equivariant fundamental group $A(\mathscr{O}_x)$ as found in \cite{NilOrbBk}. To simplify our formulas, let

$a =$ the number of distinct odd parts in $P$,

$b =$ the number of (nonzero) distinct even parts in $P$, and

$c =$ the greatest common divisor of all parts in $P$.\\
Also, a partition is called \textit{rather odd} if all of its odd parts have multiplicity one.  Notice, a very even partition is trivially rather odd.

\begin{table}[h]
	\begin{center}
		\begin{tabular}{c c c}
		\hline\hline
		\T Lie Algebra & $\pi_1(\mathscr{O}_x)$  &  $A(\mathscr{O}_x)$ \\
		\hline\hline
		\T $\mathfrak{sl}_n$  &  $\mathbb{Z}/c\mathbb{Z}$  &  1	\\
		\hline
		\T $\mathfrak{so}_{2n+1}$  &  If $P$ is rather odd, a central  &  $(\mathbb{Z}/2\mathbb{Z})^{a-1}$ \\
														&  extension by $\mathbb{Z}/2\mathbb{Z}$ of  &\\
														&  $(\mathbb{Z}/2\mathbb{Z})^{a-1}$; &\\
														&  otherwise, $(\mathbb{Z}/2\mathbb{Z})^{a-1}$ & \\
		\hline
		\T $\mathfrak{sp}_{2n}$  &  $(\mathbb{Z}/2\mathbb{Z})^{b}$  &  $(\mathbb{Z}/2\mathbb{Z})^{b}$ if all even parts\\
													&																	&	 have even multiplicity;\\
													&																	&  otherwise, $(\mathbb{Z}/2\mathbb{Z})^{b-1}$\\
		\hline
		\T $\mathfrak{so}_{2n}$  &  If $P$ is rather odd, a central & $(\mathbb{Z}/2\mathbb{Z})^{max(0,\ a-1)}$ if all \\
													&  extension by $\mathbb{Z}/2\mathbb{Z}$ of  &  odd parts have even \\
													&  $(\mathbb{Z}/2\mathbb{Z})^{max(0,\ a-1)}$;  & multiplicity; otherwise, \\
													&  otherwise, $(\mathbb{Z}/2\mathbb{Z})^{max(0,\ a-1)}$ & $(\mathbb{Z}/2\mathbb{Z})^{max(0,\ a-2)}$\\
		\hline\hline
		\end{tabular}
		\end{center}
	\caption{Equivariant Fundamental Groups}
	\label{tab:fund}
\end{table}

Although each nilpotent orbit $\Ox$ can contains multiple $T_{ad}$-orbits $\sO_{ad}(\tau_J)$, the fiber $Z(J)$ is the same for all $\sO_{ad}(\tau_J)$ in any particular $G$-orbit.  This can be seen combinatorially using Proposition \ref{ZJ} and \ref{part}, but it can also be seen from the following result.  Note that we are again viewing $W_{ad}$ as a subset of $\sN$.

\begin{prop}\label{fundprop1}
Let $x$ be in the orbit $\sO_{ad}(\tau_J)$ of $W_{ad}$.  Then, $Z(J)$ is isomorphic to the kernel of the quotient map from $\pi_1(\Ox)$ to $A(\Ox)$.
\end{prop}

\begin{proof} First, let us examine more closely the definition of $Z(J)$.  Suppose $x\in \sO_{ad}(\tau_J)$.  Then, $q: W\rightarrow W_{ad}$ is the quotient by the action of $Z$, and we can choose some $y\in W$ such that $q(y) = x$.  We know the center $Z$ of $G$ is contained in $T$ and acts on $W$. Let $y\in W$ and $Z^{y}:= \mathrm{Stab}_Z(y)$.  By definition, $Z^{y}=Z\cap T^y$.  Then $Z(J)$ is the kernel of the map $$T\cdot y \cong T/T^y\longrightarrow T\cdot x \cong T/T^x\cong T/(T^yZ).$$  This kernel is $(T^yZ)/T^y \cong Z/(Z\cap T^y) = Z/Z^y$.  Thus, the fibers $Z(J)$ of the map $q:W\rightarrow W_{ad}$ are given by $Z/Z^y$ for any $y \in q^{-1}(x)$.

Observe that the kernel $K_x$ of the quotient map $$\pi_1(\Ox)=G^{x}/(G^x)^\circ \twoheadrightarrow G^{x}/((G^x)^\circ Z)=A(\Ox)$$ is isomorphic to $Z/(Z\cap(G^x)^\circ)$ since $Z/(Z\cap(G^x)^\circ) \cong ((G^{x})^\circ Z)/(G^x)^\circ.$  Now, we will construct a surjective map from $Z(J)$ to $K_x$.  Because of the way $x$ and $y$ are defined, we know $T^y\subseteq T^x$, and moreover, since the dimensions of the $T$-orbits of $x$ and $y$ are the same, we know that $T^y$ must be a collection of components of $T^x$.  Then, $T^y$ is connected since $y\in W$ and $W$ is a toric variety for $T$.  (See \cite{Val}.) Hence, $T^y= (T^x)^{\circ}$ and thus, $Z^{y} = Z\cap (T^x)^{\circ}$.  Then, this leads to a surjective map
$$Z(J)=Z/Z^y = Z/(Z\cap(T^x)^\circ) \twoheadrightarrow Z/(Z\cap(G^x)^\circ)=K_x.$$

To complete this proof, we must show that this surjective map is actually an isomorphism.  Now, we need only show that these groups have the same size.  

Consider first type $A_n$.  Let $J = \{d_1, d_2,\ldots , d_r\}$.  Then we know $Z(J) = \Z/c\Z$ where $c = \gcd\{d_1, d_2,\ldots , d_r, n+1\}$.  From Proposition \ref{ZJ} and Proposition \ref{part}, we see that $c$ is also the greatest common divisor of the parts in ${P_J}$.  From \cite{NilOrbBk}, Corollary 6.1.6, we see that $\pi_1(\Ox) = \Z/c\Z$ as well.  Note that since $A(\Ox)$ is trivial in this type, the kernel here is all of $\pi_1(\Ox)$.

Let us denote the multiplicity of a part $d$ in a partition $P_J$ to be $\multi_{P_J}(d)$. Let us now suppose we are in type $B_n$.  Then ${P_J}$ can be of two forms.  If ${P_J}$ has exactly one odd part $d$ with $\multi_{P_J}(d) = 1$, then $Z(J) = \mathbb{Z}/2\mathbb{Z}$ since all elements of $J$ must be even if all parts with even multiplicity are even.  If ${P_J}$ has more than one odd part or an odd part with multiplicity greater than one, then $Z(J) = \{1\}$ since this can only happen when not all elements in $J$ are even.

Let us now consider type $C_n$.  Again, ${P_J}$ has two possible forms.  If ${P_J}$ is such that $\multi_{P_J}(d)$ is even for all parts $d$, then we know $n \in J$ and $Z(J) = \{1\}$.  If ${P_J}$ has exactly one part with odd multiplicity, then we know $n \notin J$ and $Z(J) = \mathbb{Z}/2\mathbb{Z}$.

Let us finally suppose we are in type $D_n$.  In this case, we have four possible forms for ${P_J}$.  If ${P_J}$ is such that $\multi_{P_J}(1) = 1$, $\multi_{P_J}(d) = 1$ for some odd part $d$, and the rest of the parts are even numbers with even multiplicity, then $Z(J) = Z$ since this corresponds to $J$ having all even elements and $n, n-1 \notin J$.  If ${P_J}$ is such that $\multi_{P_J}(1) = 1$ and there is some odd part $d$ with $\multi_{P_J}(d) > 1$, then $Z(J) = \mathbb{Z}/2\mathbb{Z}$ since this corresponds to not all elements of $J$ being even and $n, n-1 \notin J$.  If ${P_J}$ has all even parts with even multiplicities, then this corresponds to the third condition for $D_n$ in Theorem \ref{ZJ} and $Z(J) = \mathbb{Z}/2\mathbb{Z}$.  Lastly, if ${P_J}$ has $\multi_{P_J}(d)$ even for all parts $d$ and at least one $d$ is odd, then $Z(J) = \{1\}$.

Since the form of our partition $P$ tells us what $\pi_1(\mathscr{O}_x)$ and $A(\mathscr{O}_x)$ are in Table \ref{tab:fund}, we see from the above reasoning that $Z(J)$ is isomorphic to the kernel of the quotient map from $\pi_1(\mathscr{O}_x)$ to $A(\mathscr{O}_x)$.  For types $E_6$ and $E_7$, we can examine Tables \ref{tab:E6data}, \ref{tab:E7data}, and \ref{tab:E7data2} in the Appendix to see that the same is true in these types for any orbit $\Ox$ with $x\in W_{ad}$.
\end{proof}

Now that we see $Z(J)$ depends on the orbit of $x$ in $\sN$, not in $W_{ad}$, let us modify our notation to be $Z(J_x)$ to recall this fact.  Moreover, in type $A_n$, every $x\in \sN$ is conjugate to an element of $W_{ad}$, so we can define $Z(J_x)$ for any element in $\sN$ in this case to be the kernel of the quotient map above.  The fact that $A(\Ox)$ is trivial in type $A_n$ allows us to also make the following statement.

\begin{cor}\label{zpi}
If we are in type $A_n$ and $x\in W_{ad}$, then $Z(J_x)$ is isomorphic to $\pi_1(\Ox)$.
\end{cor}

Let us now extend our discussion to the map $\gamma:\sMt\rightarrow \sNt$.  First, recall the construction of $\widetilde{\ku}$.   We have $B$-equivariant maps $q: W\rightarrow W_{ad}=W/Z$ and $p: \ku\rightarrow W_{ad} = \ku/[\ku,\ku]$, and we define $\widetilde{\ku}$ to be the fiber product $W\times_{W_{ad}}\ku$.  Thus, we have the following $B$-equivariant Cartesian diagram:
\begin{center}
$\begin{CD}
\widetilde{\ku} @>\overline{q}>> \ku\\
@V\overline{p} VV @VVp V\\
W @>q >> W_{ad}
\end{CD}$
\end{center}

Now, let us restrict this picture to a single $T_{ad}$-orbit in $W_{ad}$.  Let $\sO(\tau)$ be the $T$-orbit in $W$ corresponding to the face $\tau$ of a cone for $W$, and $\sO_{ad}(\tau)$ denote the corresponding orbit in $W_{ad}$.  Define $\ku(\tau)$ to be $p^{-1}(\sO_{ad}(\tau)) \subset \ku$.  Then, we can prove the following general statement, which will be particularly useful in our analysis of the fibers in the $A_n$ case later.
\begin{prop}\label{cover}
Let us suppose $x\in \sO_{ad}(\tau)$.  Then, over the space $G\times^{B} \ku(\tau)$, $\gamma$ is a covering map with fiber $Z(J_x)$.
\end{prop}

\begin{proof}
 From \cite{Graham}, we know the map $q_\tau:\sO(\tau)\rightarrow \sO_{ad}(\tau)$ is a covering map with fibers given by $Z(J_x)$ where $x\in \sO_{ad}(\tau)$.  Graham's proof of this is simply noting that the orbits are actually tori and the map here corresponds to the map between the tori.  To simplify the notation in our diagram, let us define $\ku(\tau) := p^{-1}(\sO_{ad}(\tau)) \subset \ku$ and $\widetilde{\ku}(\tau) = \overline{p}^{-1}(\sO(\tau))\subset \widetilde{\ku}$.  We will use a $\tau$ subscript to denote the restrictions of the maps.  Now, our restricted diagram becomes:
\begin{center}
$\begin{CD}
\widetilde{\ku}(\tau) @>\overline{q}_\tau>> \ku(\tau)\\
@V\overline{p}_\tau VV @VVp_\tau V\\
\sO(\tau) @>q_\tau >> \sO_{ad}(\tau)
\end{CD}$
\end{center}
We know that since $q_\tau$ is a covering map with fibers $Z(J_x)$, $\overline{q}_\tau$ is as well.  Thus, the induced map $\widehat{{q}_\tau}:G\times \widetilde{\ku}(\tau) \rightarrow G\times \ku(\tau)$ is a covering map with these fibers.  We then have the following diagram:
\begin{center}
$\begin{CD}
G\times \widetilde{\ku}(\tau) @>\widehat{q}_\tau>> G\times \ku(\tau)\\
@V VV @VV V\\
G\times^{B} \widetilde{\ku}(\tau) @>\gamma_\tau >> G\times^{B} \ku(\tau)
\end{CD}$
\end{center}
Since $\widehat{{q}_\tau}$ is a covering map with fibers $Z(J_x)$, so is the map $\gamma_\tau$.  
 \end{proof}

 
\section{Some Irreducible Components of Graham's Fibers in Type $A_n$}

For our perverse sheaf calculations, we will need to know more about the top degree cohomology of the Graham fibers $\Mx$.  More specifically, since the irreducible components of $\Mx$ parametrize a basis for this cohomology, we need to know more about the irreducible components of $\Mx$.  While we will not obtain a complete picture of these irreducible components here, we will ascertain enough to make the desired perverse sheaf statement.  In particular, we will use a combinatorially described affine paving for the Springer fibers in type $A_n$ to find a particular affine space $\sA_x$ of maximal dimension inside $\Bx$ and then show that over $\sA_x$, Graham's map $\gamma$ is a covering map with fibers $\pi_1(\Ox)$.   The specific description we use for this affine paving follows Tymoczko \cite{TymLin}, where her result is a generalization of this to Hessenberg varieties.  However, the original result for Springer fibers in type $A_n$ is due to Shimomura \cite{shim}.

%
%
%

Since these results are type-specific, let us assume $\sN$ is the nilpotent cone for $\ksl_n(\mathbb{C})$ and is thus the set of nilpotent $n$ by $n$ complex matrices with trace zero.  Let $x\in \sN$ be a matrix in Jordan canonical form. Let $B$ be the Borel subgroup of $G$ consisting of linear transformations represented by upper triangular matrices. Thus, if we were to follow Graham's construction of $\sMt$ using $B$ as our chosen Borel subgroup, we would have $x\in W_{ad}$ since it is the sum of simple root vectors for $B$.  Let $\sO_{ad}(\tau_J)$ denote the $T_{ad}$-orbit containing $x$ in $W_{ad}$.  Let $P = [d_r\ d_{r-1}\ \ldots \ d_1]$ with $d_i\geq d_{i-1}$ be the partition of $n$ corresponding to $x$ and denote by $Y_P$ the Young diagrams for $P$.  Here, we follow the convention that $Y_P$ is the left-justified array of boxes where the $i$th row from the bottom has $d_i$ boxes.  Next, we will describe two different labellings of this diagram.  For the first, fill the blocks with $\{1, 2,\ldots ,n\}$ in increasing order starting at the bottom left and moving up the columns, treating the columns left to right.  We will call $Y_P$ with this labelling $Y_P^{Tym}$.  For the second, again label $Y_P$ with $\{1, 2,\ldots ,n\}$ in increasing order, but this time start at the top left and fill in the rows.  We will call $Y_P$ with this labelling $Y_P^{Std}$.  Let $\w$ be the permutation taking $Y_P^{Std}$ to $Y_P^{Tym}$, i.e $\w(j)$ is the number in $Y_P^{Tym}$ occupying the same box that $j$ occupies in $Y_P^{Std}$.\\


\begin{example}
Let us suppose we are in $\ksl_5(\mathbb{C})$.  Let $x\in \sN$ be an element in the orbit with partition $P = [2\ 2\ 1]$.  Then

\begin{multicols}{3}

 $x = \begin{bmatrix}

0 & 1 & 0 & 0 & 0 \\
0 & 0 & 0 & 0 & 0 \\
0 & 0 & 0 & 1 & 0 \\
0 & 0 & 0 & 0 & 0 \\
0 & 0 & 0 & 0 & 0 \\

\end{bmatrix}$,

\columnbreak
$Y_P^{Tym}$ =  \ytableausetup{centertableaux} \begin{ytableau}
3 & 5  \\ 2 & 4 \\ 1
\end{ytableau} ,\\

\columnbreak

$Y_P^{Std}$= \ytableausetup{centertableaux} \begin{ytableau}
1 & 2  \\ 3 & 4  \\ 5
\end{ytableau} ,\\

\end{multicols}
and $\w= (1\ 3\ 2\ 5)$.\\
\end{example}

Let us call any two adjacent boxes in the same row of a labelled diagram a \textit{pair} and use the notation $(i|j)$ to mean a pair where the label in the left box is $i$ and the label in the right one is $j$.  We can use the above labellings of $Y_P$ to form nilpotent matrices by placing a $1$ in the $i$th row and $j$th column for every pair $(i|j)$ in the labelled diagram and by filling the remaining entries with $0$'s.  (In other words, this is the sum of all elementary matrices $E_{i,j}$ where $(i|j)$ is a pair.)  The matrix obtained this way from $Y_P^{Std}$ is $x$.  The matrix $\overline{x}$ obtained this way from $Y_P^{Tym}$ is not in Jordan canonical form typically, but rather what Tymoczko calls \textit{highest form}. Furthermore, $\overline{x} = \sigma\cdot x$ where the action is conjugation by the appropriate permutation matrix. 

In \cite{TymLin}, Tymoczko describes an affine paving of the Springer fibers using intersections with Schubert cells.  Since the Schubert cells decompose the flag variety, we need to realize our Springer fibers as subsets of the flag variety.  We defined $\sNt$ as $G\times^B\ku$ since this definition lets us most clearly see our map $\gamma:\sMt \rightarrow \sNt$.  However, we also have \[G\times^B\ku \cong \{(gB, u)\in G/B\times \sN\colon g^{-1}\cdot u \in \ku\}\]  which leads to the Springer fiber description \[\mu^{-1}(u) = \{(gB,u)\in G/B\times\{u\}\colon g^{-1}\cdot u \in \ku\} \cong \{gB\in G/B \colon g^{-1}\cdot u\in \Lie(B)\}.\]  Note that here we use $u$ as our general nilpotent element since $x$ has been previously fixed.

Define $\sA_x := \hB \w \hB \cap \mu^{-1}(\overline{x})$ where $\hB \w \hB$ is the Schubert cell associated to $\w$.  For this to make sense, we are viewing $\mu^{-1}(\overline{x})$ here really as its image in $G/B$.  Thus, $\sA_x$ is a subset of $G/B$. However, just as $\mu^{-1}(\overline{x})$ is contained in $\sNt$, so is $\sA_x$.  For this, we use the definition of $B$ given above to see \[\sA_x= \{ (\overline{x},b\sigma B) \in \{\overline{x}\}\times \mathscr{B}\ | b\in B\mathrm{\ and\ } (b\sigma)^{-1}\overline{x}\in \Lie B\}\] which is contained in \[\sNt = \{ (u, B')\in \sN \times \mathscr{B}\ |\ u\in \sN \mathrm{\ and\ }u\in \Lie B'\}.\] In the future, it should be clear from the context whether we are viewing $\sA_x$ as in $G/B$ or in $\sNt$.

For any $w \in \W$, we will denote by $\Phi_{w}$ the set of positive roots for $B$ that become negative under the action of $w^{-1}$.  For any $u \in \sN$, we will let $\Phi_u$ be the positive roots for $B$ whose vectors appear as summands when $u$ is decomposed as the sum of root vectors.  We will denote by $\Phi_{w,u}$ the subset of $\Phi_{w}$ whose elements can be viewed as a sum of a collection of roots in $\Phi_{w}$ and a root in $\Phi_u$.  We have chosen $\hB$ and $Y_P^{Tym}$ such that the relevant parts of Tymoczko's Corollary 6.3 in \cite{TymLin} can be phrased as follows in the special case that the Hessenberg space is $\Lie \hB$ and the corresponding Hessenberg variety is $\mu^{-1}(\overline{x})$.

\begin{thm}[Shimomura, Tymoczko] \label{Tymoczko}

(a) Let $w\in \W$.  The Schubert cells $\hB w \hB$ intersect each Springer fiber in a paving by affines.  The nonempty intersections are $\hB w\hB\cap \mu^{-1}(\overline{x})$ where $\Ad(w^{-1}) x \in \Lie \hB$, and they have dimension
$$ |\Phi_w| - |\Phi_{w,\overline{x}}|.$$

\noindent(b) The nonempty cells from (a) correspond to the permutations $w$ such that $w^{-1}(Y_P^{Tym})$ has the property that $i<j$ for each pair $(i|j)$.  

\end{thm}

We can now use our labelled Young diagrams $Y_P^{Tym}$ and $Y_P^{Std} = \w^{-1}(Y_P^{Tym})$ with the above theorem to find our affine space of maximal dimension in $\mu^{-1}(\overline{x})$.

\begin{lem}\label{affinelem}
$\sA_x$ is nonempty, and it is an affine space of dimension $\frac{1}{2}(\dim\sN-\dim\Ox)=\dim\mu^{-1}(\overline{x})$.
\end{lem}

\begin{proof}

To see that $\sA_x$ is nonempty, we will need only to note that $Y_P^{Std}$ is filled with labels that increase from left to right.  Thus, the condition in Theorem \ref{Tymoczko} (b) is always satisfied.   From the discussions in previous paragraphs, $\Phi_{\overline{x}} = \{\alpha_{i,j}|$ $(i|j)$ is a pair in $Y_P^{Tym}\}$.  Given our description of the positive roots of $B$, we can restate our definition of $\Phi_{\w}$ as $\alpha_{i,j}\in \Phi_{\w}$ for $i<j$ if and only if $\w^{-1}(i) > \w^{-1}(j)$.  Since we used our labellings of $Y_P$ to define $\w$, $\Phi_{\w}$ can be seen from $Y_P^{Tym}$.  Let us number the rows in $Y_P^{Tym}$ in increasing order from top to bottom and number the columns in increasing order from left to right.  For any label $i$ in $Y_P^{Tym}$, let $\row(i)$ be the row number of the row containing $i$ and $\col(i)$ be the column number of the column containing $i$.  Using this notation, $\alpha_{i,j}$ is in $\Phi_{\w}$ if and only if $\row(i) > \row(j)$ and $\col(i) \geq \col(j)$.  To see this, consider the box in $Y_P^{Tym}$ containing $i$.  All the numbers in the boxes in higher rows or columns are larger than $i$, but the opposite fact is true of that box in $Y_P^{Std}$.  We can also see $\Phi_{\w,\overline{x}}$ from $Y_P^{Tym}$. In particular, $\alpha_{i,j}$ is in $\Phi_{\w,\overline{x}}$ if and only if $\row(i) > \row(j)$ and $\col(i) > \col(j)$.  See Example \ref{ex2}.

Viewing $\Phi_{\w}$ and $\Phi_{\w,\overline{x}}$ in this manner allows us to translate the formula from Theorem \ref{Tymoczko}(a) into a statement about the number of blocks in the columns of $Y_P$.  More specifically, if we let $\height(j)$ denote the number of blocks (or height) in the $j$th column of $Y_P$, then $$|\Phi_{\w}| - |\Phi_{\w,\overline{x}}| = \sum_{j=1}^{d_r}\sum_{i=1}^{\height(j)} (\height(j) - i) $$
where $d_r$ is the largest part in the partition $P$ and thus the number of columns in $Y_P$.  From \cite{NilOrbBk} Corollary 6.1.4, we see that the dimension of $\Ox$ is $(n+1)^2-\sum_{j=1}^{d_r} \height(j)^2$.  Thus, we have 

\begin{align*}
\dim \sN - 2\dim \sA_x & = n(n+1) - 2 \sum_{j=1}^{d_r}\sum_{i=1}^{\height(j)} \left( \height(j) - i\right) \\
& = n^2 + n - 2\sum_{j=1}^{d_r} \left(\height (j)^2 - \left(\frac{\height (j)^2 + \height (j)}{2}\right)\right)\\
& = n^2 + n - \sum_{j=1}^{d_r} \height (j)^2 + \sum_{j=1}^{d_r} \height (j)\\
& = (n+1)^2 - \sum_{j=1}^{d_r} \height (j)^2 \\
&= \dim \Ox.\\
\end{align*}
Since, the dimension of $\sA_x$ is equal to half the codimension of the orbit $\Ox$, we see that $\sA_x$ is of maximal possible dimension.
\end{proof}

\begin{example}\label{ex2}
Suppose $x$ corresponds to the partition ${P = [3\ 3\ 1]}$. Then\\ 
\begin{center}
$Y_P^{Std}$ =  \ytableausetup{centertableaux} \begin{ytableau}
1 & 2 & 3  \\ 4 & 5 & 6 \\ 7
\end{ytableau}\ ,\ \ 
$Y_P^{Tym}$ =  \ytableausetup{centertableaux} \begin{ytableau}
3 & 5 & 7  \\ 2 & 4 & 6 \\ 1
\end{ytableau}\ ,
\end{center}

\noindent$\Phi_{\overline{x}} = \{\alpha_{2,4},\  \alpha_{3,5},\ \alpha_{4,6},\ \alpha_{5,7}\}$,\\
$\Phi_{\w} = \{\alpha_{1,2},\ \alpha_{1,3},\ \alpha_{1,4},\ \alpha_{1,5},\ \alpha_{1,6},\ \alpha_{1,7},\ \alpha_{2,3},\ \alpha_{2,5},\  \alpha_{2,7},\ \alpha_{4,5},\ \alpha_{4,7},\ \alpha_{6,7}\}$,\\
and $\Phi_{\w,\overline{x}} = \{\alpha_{1,4},\ \alpha_{1,6},\ \alpha_{1,5},\ \alpha_{1,7},\ \alpha_{2,5},\ \alpha_{2,7},\ \alpha_{4,7}\}$.\\

\noindent Thus, $\dim \sA_x = |\Phi_{\w}| - |\Phi_{\w,\overline{x}}| = 5$.  Notice that $\dim \Ox = 32$ and $\dim \sN = 42$ so this agrees with the above calculations.
\end{example}

\begin{prop}\label{coveringmap}
Graham's map $\gamma$ is a covering map over $\sA_x$ with fibers given by $Z(J_x)$.
\end{prop}

\begin{proof}
First, note that we can construct Graham's variety based on the choice of any Borel subgroup, so we will assume this has been done for the specific $B$ described at the start of this section based on our choice of $x$.  Let $\sO_{ad}(\tau)$ be the $T_{ad}$ orbit containing $x$ in $W_{ad}$.  From Proposition \ref{cover}, we've identified that $\gamma_\tau$ is a covering map over a specific subspace of $\sNt$, so that we need now only show that $\sA_x \subseteq G\times^{B} \ku(\tau)$. 

As mentioned before, we can view $\sA_x$ as contained in $\sNt$ since \[\sA_x= \{ (\overline{x},b\sigma B) \in \{\overline{x}\}\times \mathscr{B}\ | b\in B\mathrm{\ and\ } (b\sigma)^{-1}\cdot\overline{x}\in \Lie B\}.\] Thus, to see $\sA_x\subseteq G\times^{B} \ku(\tau)$, we need to see that $\sA_x$ is contained in the image of $G\times^{B} \ku(\tau)$ under the isomorphism $G\times^{B} \ku\rightarrow \sNt$ given by $(g,y)\mapsto (g\dot y,gB)$. Using this isomorphism, $\sA_x$ can be identified as  \[\{(\sigma \cdot b^{-1} \cdot \overline{x}, b\sigma  B) | b\in B\mathrm{\ and\ }\sigma^{-1}\cdot b^{-1}\overline{x} \in \ku \}.\] Then, we need to see that $\sigma^{-1}\cdot b^{-1}\cdot\overline{x}$ decomposes to a sum of root vectors where the simple root vectors correspond to $\tau$, which was determined by the simple root vectors in the decomposition of $x$.

Since $b\in B$, we know \[b^{-1}\cdot \overline{x} = \overline{x} + \sum_{\rho>\beta}c_{\rho}X_{\rho}\] where $\beta \in \Phi_{\overline{x}}$ and $\rho$ are positive roots.  Here, $\gamma > \beta$ means $\gamma - \beta$ is a sum of positive roots.  Then we have \[\sigma^{-1}\cdot b^{-1}\cdot\overline{x} = \sigma^{-1}\cdot \left(\overline{x} + \sum_{\rho>\beta}c_{\rho}X_{\rho}\right) = x + \sum_{\rho>\beta}c_{\rho}X_{\sigma^{-1}(\rho)}.\]
Now, we need to verify that $\sigma^{-1}(\rho)$ is not a simple root for $c_\rho\neq 0$.  Then, we would have that our $\sigma^{-1}\cdot b^{-1}\cdot\overline{x}$ decomposes as the sum of simple root vectors of $x$ and other positive, non-simple root vectors, thus giving us $\sA_x\subseteq G\times^{B} \ku(\tau)$. Since we know $\sigma^{-1}\cdot b^{-1}\cdot\overline{x}\in \ku$, we have that $\sigma^{-1} (\rho)$ must be a positive root for all $c_\rho\neq 0$.  We know $\rho - \beta$ is a sum of positive roots, and thus, \[\sigma^{-1} (\rho) = \sigma^{-1}(\beta) + \sum_\alpha \sigma^{-1}(\alpha) \] for some positive roots $\alpha$.  We know $\sigma^{-1}(\beta)$ is simple since $\beta$ appears in $\Phi_{\overline{x}} = \sigma(\Phi_x)$.  Also, since we know $\sigma^{-1}(\rho)$ is positive, we have $\sum_\alpha \sigma^{-1}(\alpha)$ is a sum of positive roots again. If $\sigma^{-1}(\rho)$ were simple, the equation above would give a decomposition of the simple root into the sum of a simple root and other positive roots, which is a contradiction.
\end{proof}

As mentioned at the start of this section, these propositions do not give a full understanding of Graham's fibers, or even of the irreducible components.  However, they do give us enough information that we can make statements in the perverse sheaf setting.

\section{Main Result}

\subsection{Borho and MacPherson's Results}

Let $\mathscr{P}(\sN)$ be the category of perverse sheaves on $\sN$ with respect to the stratification by $G$-orbits.  Here again, $\kg$ is a complex semisimple Lie algebra and $G$ is the simply-connected algebraic group associated to it.  Define $d_x := \dim \Bx$.  Let $\underline{\Qlb}_{\sNt}$ denote the constant $\ell$-adic sheaf on $\sNt$.  For any orbit $\Ox$, let $\mathscr{L}_{\varphi}$ be the local system corresponding to the representation $\varphi$ of $\pi_1(\Ox)$.

All the results in the following fact can be seen in \cite{Borho1983}, but they were first proven elsewhere.

\begin{fact}\label{Spr}
\begin{enumerate}[(a)]
\item The Springer resolution is proper and semismall.
\item $\sNt$ is smooth, thus rationally smooth.
\item $\dim \sNt = \dim \sN$.
\item $2d_x = \dim \sN - \dim \Ox$ for any $x \in \sN$.
\end{enumerate}
\end{fact}

The above facts tell us that the Springer resolution satisfies all the necessary conditions to apply the Decomposition Theorem for perverse sheaves from \cite{BBD}; see also \cite{Borho1983}.   Thus, we know that $R\mu_*\underline{\Qlb}_{\sNt}$ is semisimple in $\mathscr{P}(\sN)$.  Each simple perverse sheaf that occurs as a summand corresponds to a local system $\mathscr{L}_{\varphi}$ on an orbit $\Ox$.  Here, $\varphi$ is an irreducible representation of the fundamental group $\pi_1(\Ox)$, and the $\varphi$ which appear are exactly those which occur in the action of $\pi_1(\Ox)$ on $H^{2d_x}(\Bx)$.  In \cite{Borho1983}, Borho and MacPherson use this in their proof of the Springer correspondence by noting that the local systems which appear  same ones previously constructed by Springer, and furthermore, that the multiplicity with which they appear is the same as the dimension of the representation of the Weyl group to which they correspond.  This means that in type $A_n$, the trivial local system, and only the trivial local system, appears for each orbit.  We will see below that more local systems occur for Graham's variety.

\begin{remark}
Although Borho and MacPherson state in \cite{Borho1983} that the monodromy representation of $\pi_1(\Ox)$ on $H^{2d_x}(\Bx)$ is the same up to multiplication by the sign representation as the action of the component group given by Springer in \cite{Springer1976}, their argument is not very explicit.  The discussion by Jantzen in \cite{Jantzen2004} is more thorough, but rather than argue directly about the actions, Jantzen focuses on how the local systems appearing are $G$-equivariant and can be determined by the action of the component group $G^x/(G^x)^\circ$ on $H^{2d_x}(\Bx)$.  As $\pi_1(\Ox)$ is viewed as the component group in the whole of this paper, we will follow the same approach as Jantzen. 
\end{remark}

\subsection{Results for Graham's Variety}

\begin{prop}\label{decomp}
The map $\tilde{\mu}$ is proper and semismall.  $\sMt$ is rationally smooth with $\dim \sMt=\dim \sN$.  Furthermore, $2d_x= \dim \sN - \dim \Ox$ for any $x\in \sN$.
\end{prop}

\begin{proof}
According to \cite{Graham}, $\gamma$ is finite and $\widetilde{\mathscr{M}}$ is locally the quotient of a smooth variety by a finite group.  Thus, $\sMt$ is rationally smooth \cite{Brion1999}, and all of the properties except rational smoothness follow from the finiteness of $\gamma$ and Fact \ref{Spr}. 
\end{proof}

Now, the results from Section 3.2 are enough to make the following useful proposition: 

\begin{prop}\label{regrep}
For $\kg$ in type $A_n$, $H^{2d_x}(\Mx)$ contains a copy of the regular representation of $\pi_1(\Ox)$.
\end{prop}

\begin{proof}
There is a basis of $H^{2d_x}(\Mx)$ indexed by irreducible components of maximal dimension.  Therefore, determining how $\pi_1(\Ox)$ acts on the components is sufficient to understand how it acts on the cohomology.  We know $\pi_1(\Ox) = G^x/(G^x)^\circ$ acts on these components through the same argument that holds for Springer fibers.  That is, $G$ acts on $\sMt$, and thus $G^x$ acts on the fiber $\Mx$.  Since the identity component $(G^x)^\circ$ is irreducible and connected, so are the orbits from the action of this subgroup.  Thus, the irreducible components of $\Mx$ are preserved by the action of $G^x/(G^x)^\circ$, and we know $\pi_1(\Ox)$ must act by permuting the irreducible components of $\Mx$.  Now, the goal is to understand as much as possible about this action.   

From Corollary \ref{zpi} and Proposition \ref{coveringmap}, we know that $\gamma$ is a covering map over $\sA_x$ with  fibers given by $\pi_1(\Ox)$.  We further know that $\sA_x$ is simply-connected since Lemma \ref{affinelem} says it is an affine space.  Therefore, $\gamma^{-1}(\sA_x)$ must be a disjoint union of ``copies'' of $\sA_x$. Since $\dim \sA_x$ is maximal in $\Bx$, it must be the case that the closure of $\sA_x$ is an irreducible component of maximal dimension in $\Bx$, and similarly, the closure of each copy of $\sA_x$ found in $\gamma^{-1}(\sA_x)$ must be an irreducible component of maximal dimension in $\Mx$.  Consider the fiber over any point in $\sA_x$.  This is a finite set and $Z(J_x)$ acts freely and transitively on the points in the fiber.  This action is given by left multiplication and must agree with the left multiplication action of $\pi_1(\Ox)$ since we saw in the proof of Proposition \ref{fundprop1} that each coset in $G^x/(G^x)^\circ$ can be represented with an element of $Z$.  Thus, we know that $\pi_1(\Ox)$ acts freely and transitively on these points, so also on the copies of $\sA_x$ contained in $\Mx$.  Therefore, it acts freely and transitively on the irreducible components contained in $\overline{\gamma^{-1}(\sA_x)}$.  We see then that there must be a copy of the regular representation of $\pi_1(\Ox)$ contained in its action on $H^{2d_x}(\Mx)$.
\end{proof}

With these propositions in place, we can now state our main result. 

\begin{thm}
If $\kg$ is a Lie algebra of type $A_n$,  then every $G$-equivariant simple perverse sheaf occurs as a summand in the decomposition of $R\mu_*\underline{\Qlb}_{\sMt}$.
\end{thm}

\begin{proof}
From Proposition \ref{decomp}, we know that the Decomposition Theorem applies.  Then, since $\tilde{\mu}$ is semismall, the Decomposition Theorem tells us that $\IC (\Ox, \mathscr{L}_{\varphi})$ occurs in $R\mu_*\underline{\Qlb}_{\sMt}$ if and only if the representation $\varphi$ occurs in the action of  $\pi_1(\Ox)$ on $H^{2d_x}(\Mx)$.  From Proposition \ref{regrep}, we see that all the irreducible representations appear in this action.  The result follows.

\end{proof}

\section{Appendix: Exceptional Types $E_6$ and $E_7$}

The Bala-Carter classification was used to determine the containment of $T_{ad}$-orbits inside $G$-orbits for the exceptional Lie algebras $E_6$ and $E_7$.  In the case of $E_7$, it was also necessary to calculate the weighted Dynkin diagrams when the same type of Levi subalgebra labels more than one orbit. 

The following tables do not include the $G$-orbits which contain no $T_{ad}$-orbits.  For any orbit $\sO$, the $G_{ad}$-equivariant fundamental group $A(\sO)$ can be determined from the $\pi_1(\sO)$ column by taking the quotient by any $\Z/2\Z$ or $\Z/3\Z$ factor present.  The values for $\pi_1(\sO)$ and $A(\sO)$ come from \cite{Alexeevski2005} and \cite{Mizuno80}.  They can also be found in \cite{NilOrbBk} Section 8.4, but with some errors in the results for $E_7$.  

\begin{table}[h]
\begin{center}
\begin{tabular}{c p{2.5in} cc}
\hline
Bala--Carter \T \B & Subsets $J$ corresponding to $T_{ad}$-orbits & $Z(\mathscr{O})$ & $\pi_1(\mathscr{O})$\\
\hline\hline
Triv. \T \B & $\{1,2,3,4,5,6\}$ & 1 & 1\\
\hline
$A_1$ \T & $\{1,2,3,4,5\}$, $\{1,2,3,4,6\}$, $\{1,2,3,5,6\}$,& 1 & 1\\
 \B  & $\{1,2,4,5,6\}$, $\{1,3,4,5,6\}$, $\{2,3,4,5,6\}$ & & \\
      \hline
$2A_1$ \T & $\{1,2,3,5\}$, $\{1,2,4,5\}$, $\{1,2,4,6\}$, & 1 & 1\\
 & $\{1,2,5,6\}$, $\{1,3,4,5\}$, $\{1,3,4,6\}$, & & \\
 &  $\{1,4,5,6\}$, $\{2,3,4,5\}$, $\{2,3,4,6\}$, & & \\
\B &  $\{2,3,5,6\}$, $\{3,4,5,6\}$ & & \\
    \hline
$3A_1$ \T  & $\{1,4,5\}$, $\{1,4,6\}$, $\{2,3,5\}$, & 1 & 1\\
\B &  $\{3,4,5\}$, $\{3,4,6\}$ & & \\
\hline
$A_2$ \T  & $\{1,2,3,4\}$, $\{1,2,3,6\}$, $\{1,3,5,6\}$,  & 1 & $S_2$\\
\B & $\{2,4,5,6\}$  & & \\
 \hline
$A_2 + A_1$ \T & $\{1,2,4\}$, $\{1,2,5\}$, $\{1,3,4\}$, $\{1,3,5\}$, & 1 & 1 \\
 & $\{2,3,4\}$, $\{2,3,6\}$, $\{2,4,5\}$, $\{2,4,6\}$,  & & \\
\B & $\{3,5,6\}$, $\{4,5,6\}$ & & \\
    \hline
$2A_2$  \T \B & $\{2,4\}$ & $\Z/3\Z$ & $\Z/3\Z$ \\
\hline
$A_2 + 2A_1$ \T \B & $\{1,4\}$, $\{3,4\}$, $\{3,5\}$, $\{4,5\}$, $\{4,6\}$ & 1 &1\\
\hline
$A_3$ \T & $\{1,2,3\}$, $\{1,2,6\}$, $\{1,3,6\}$, & 1 & 1\\
\B & $\{1,5,6\}$, $\{2,5,6\}$ & & \\
\hline
$2A_2 + A_1$ \T \B & $\{4\}$ & $\Z/3\Z$ & $\Z/3\Z$\\
\hline
$A_3 + A_1$ \T \B & $\{1,5\}$, $\{2,3\}$, $\{2,5\}$, $\{3,6\}$ & 1 & 1\\
\hline
$A_4$ \T \B & $\{1,2\}$, $\{1,3\}$, $\{2,6\}$, $\{5,6\}$ & 1 & 1\\
\hline
$D_4$ \T \B & $\{1,6\}$ & 1 & 1\\
\hline
$A_4 + A_1$ \T \B & $\{3\}$, $\{5\}$ & 1 & 1\\
\hline
$A_5$ \T \B & $\{2\}$ & $\Z/3\Z$ & $\Z/3\Z$\\
\hline
$D_5$ \T \B & $\{1\}$, $\{6\}$ & 1 & 1\\
\hline
$E_6$ \T \B & $\varnothing$ & $\Z/3\Z$ & $\Z/3\Z$\\
\hline\hline
\end{tabular}
\end{center}
\caption{$T_{ad}$-orbits and Fundamental Groups for $E_6$}
\label{tab:E6data}
\end{table}

\clearpage

\begin{table}[ht]
\begin{center}
\begin{tabular}{c p{2.5in} cc}
\hline
Bala--Carter  \T \B & Subsets $J$ corresponding to $T_{ad}$-orbits & $Z(\mathscr{O})$ & $\pi_1(\mathscr{O})$\\
\hline\hline
Triv. \T \B & $\{1,2,3,4,5,6,7\}$ & 1 & 1\\
\hline
$A_1$ \T & $\{1,2,3,4,5,6\}$, $\{1,2,3,4,5,7\}$, & 1 & 1\\
  & $\{1,2,3,4,6,7\}$, $\{1,2,3,5,6,7\}$, & & \\
  & $\{1,2,4,5,6,7\}$, $\{1,3,4,5,6,7\}$, & & \\
   \B   & $\{2,3,4,5,6,7\}$ & & \\
      \hline
$2A_1$ \T & $\{1,2,3,4,6\}$, $\{1,2,3,5,6\}$, $\{1,2,3,5,7\}$, & 1 & 1\\
 & $\{1,2,4,5,6\}$, $\{1,2,4,5,7\}$, $\{1,2,4,6,7\}$, & & \\
  & $\{1,3,4,5,6\}$, $\{1,3,4,5,7\}$, $\{1,3,4,6,7\}$, & & \\
   & $\{1,4,5,6,7\}$, $\{2,3,4,5,6\}$, $\{2,3,4,5,7\}$, & & \\
 \B   & $\{2,3,4,6,7\}$, $\{2,4,5,6,7\}$, $\{3,4,5,6,7\}$ & & \\
    \hline
$(3A_1)''$ \T \B & $\{1,3,4,6\}$ & $\mathbb{Z}/2\mathbb{Z}$ & $\mathbb{Z}/2\mathbb{Z}$\\
\hline
$(3A_1)'$ \T & $\{1,2,4,6\}$, $\{1,4,5,6\}$, $\{1,4,5,7\}$, & 1 & 1\\
 & $\{1,4,6,7\}$, $\{2,3,4,6\}$, $\{2,3,5,6\}$, & &\\
 & $\{2,3,5,7\}$, $\{3,4,5,6\}$, $\{3,4,5,7\}$, & &\\
 \B & $\{3,4,6,7\}$ & & \\
  \hline
$A_2$ \T & $\{1,2,3,4,5\}$, $\{1,2,3,4,7\}$, $\{1,2,3,6,7\}$, & 1 & $S_2$\\
 \B & $\{1,2,5,6,7\}$, $\{1,3,5,6,7\}$ & &\\
 \hline
$4A_1$ \T \B & $\{1,4,6\}$, $\{3,4,6\}$ & $\mathbb{Z}/2\mathbb{Z}$ & $\mathbb{Z}/2\mathbb{Z}$ \\
\hline
$A_2 + A_1$ \T & $\{1,2,3,5\}$, $\{1,2,3,6\}$, $\{1,2,4,5\}$, & 1 & $S_2$ \\
 & $\{1,2,4,7\}$, $\{1,2,5,6\}$, $\{1,2,5,7\}$, & & \\
  & $\{1,3,4,5\}$, $\{1,3,4,7\}$, $\{1,3,5,6\}$, & & \\
  & $\{1,3,5,7\}$, $\{1,5,6,7\}$, $\{2,3,4,5\}$, & & \\
   & $\{2,3,4,7\}$, $\{2,3,6,7\}$, $\{2,4,5,6\}$, & & \\
   & $\{2,4,5,7\}$, $\{2,4,6,7\}$, $\{3,4,6,7\}$, & & \\
  \B & $\{4,5,6,7\}$ & & \\
    \hline
$A_2 + 2A_1$ \T & $\{1,4,5\}$, $\{1,4,7\}$, $\{2,3,5\}$, $\{2,3,6\}$, & 1 &1\\
 & $\{2,4,6\}$, $\{3,4,5\}$, $\{3,4,7\}$, $\{3,5,6\}$, & & \\
 & $\{3,5,7\}$, $\{4,5,6\}$, $\{4,5,7\}$, & &\\
 \B & $\{4,6,7\}$ & & \\
\hline
$2A_2$ \T \B & $\{2,4,5\}$, $\{1,3,5\}$, $\{2,4,7\}$, $\{1,2,5\}$ & 1 & 1\\
\hline
$A_2 + 3A_1$ \T \B & $\{4,6\}$ & $\mathbb{Z}/2\mathbb{Z}$ & $\mathbb{Z}/2\mathbb{Z}$\\
\hline
$A_3$ \T & $\{1,2,3,4\}$, $\{1,2,3,7\}$, $\{1,2,6,7\}$, & 1 & 1\\
 \B & $\{1,3,6,7\}$, $\{2,5,6,7\}$ & &\\
\hline
$(A_3 + A_1)''$ \T \B & $\{1,3,4\}$, $\{1,3,6\}$ & $\mathbb{Z}/2\mathbb{Z}$ & $\mathbb{Z}/2\mathbb{Z}$\\
\hline
$2A_2 + A_1$ \T \B & $\{3,5\}$, $\{4,5\}$, $\{4,7\}$ & 1 & 1\\
\hline
\end{tabular}
\end{center}
\caption{$J$-Sets and Fundamental Groups for $E_7$}
\label{tab:E7data}
\end{table}

\clearpage

\begin{table}[ht]
\begin{center}
\begin{tabular}{c p{2.5in} cc}
\hline
Bala--Carter  \T \B & Subsets $J$ corresponding to $T_{ad}$-orbits & $Z(\mathscr{O})$ & $\pi_1(\mathscr{O})$\\
\hline\hline
$(A_3 + A_1)'$ \T & $\{1,5,6\}$, $\{1,5,7\}$, $\{2,3,4\}$, $\{2,3,7\}$, & 1 & 1\\
  & $\{2,5,6\}$, $\{2,5,7\}$, $\{3,6,7\}$, $\{1,2,4\}$, & & \\
 \B & $\{1,2,6\}$ & & \\
\hline
$A_3 + 2A_1$ \T \B & $\{1,4\}$, $\{3,4\}$, $\{3,6\}$ & $\mathbb{Z}/2\mathbb{Z}$ & $\mathbb{Z}/2\mathbb{Z}$\\
\hline
$D_4$ \T \B & $\{1,6,7\}$ & 1 & 1\\
\hline
$A_3 + A_2$ \T \B & $\{1,5\}$, $\{2,4\}$, $\{2,5\}$ & 1 & $S_2$\\
\hline
$A_3 + A_2 + A_1$ \T \B & $\{4\}$ & $\mathbb{Z}/2\mathbb{Z}$ & $\mathbb{Z}/2\mathbb{Z}$\\
\hline
$A_4$ \T & $\{1,2,7\}$, $\{1,3,7\}$, $\{2,6,7\}$, $\{5,6,7\}$, & 1 & $S_2$\\
\B & $\{1,2,3\}$ & & \\
\hline
$(A_5)''$ \T \B & $\{1,3\}$ & $\mathbb{Z}/2\mathbb{Z}$ & $\mathbb{Z}/2\mathbb{Z}$\\
\hline
$D_4 + A_1$ \T \B & $\{1,6\}$ & $\mathbb{Z}/2\mathbb{Z}$ & $\mathbb{Z}/2\mathbb{Z}$\\
\hline
$A_4 + A_1$ \T \B & $\{2,3\}$, $\{2,6\}$, $\{3,7\}$, $\{5,6\}$, $\{5,7\}$ & 1 & $S_2$\\
\hline
$A_4 + A_2$ \T \B & $\{5\}$ & 1 & 1\\
\hline
$(A_5)'$ \T \B & $\{1,2\}$, $\{2,7\}$ & 1 & 1\\
\hline
$A_5 + A_1$ \T \B & $\{3\}$ & $\mathbb{Z}/2\mathbb{Z}$ & $\mathbb{Z}/2\mathbb{Z}$\\
\hline
$D_5$ \T \B & $\{1,7\}$, $\{6,7\}$ & 1 & 1\\
\hline
$A_6$ \T \B & $\{2\}$ & 1 & 1\\
\hline
$D_5 + A_1$ \T \B & $\{6\}$ & $\mathbb{Z}/2\mathbb{Z}$ & $\mathbb{Z}/2\mathbb{Z}$\\
\hline
$D_6$ \T \B & $\{1\}$ & $\mathbb{Z}/2\mathbb{Z}$ & $\mathbb{Z}/2\mathbb{Z}$\\
\hline
$E_6$ \T \B & $\{7\}$ & 1 & 1\\
\hline
$E_7$ \T \B & $\varnothing$ & $\mathbb{Z}/2\mathbb{Z}$ & $\mathbb{Z}/2\mathbb{Z}$\\
\hline\hline
\end{tabular}
\end{center}
\caption{$J$-Sets and Fundamental Groups for $E_7$}
\label{tab:E7data2}
\end{table}

\section{Acknowledgments} 

The author thanks William Graham for the use of his unpublished results and for his wealth of insight and advice on this project.  This project was completed as the author's dissertation, and thus special thanks and an abundance of gratitude are also extended to her advisor, Pramod Achar.  The author also thanks Martha Precup for her especially helpful suggestions during the editing process.

\bibliographystyle{jabrefpna}	
\bibliography{library}

\ifx\undefined\bysame
\newcommand{\bysame}{\leavevmode\hbox to3em{\hrulefill}\,}
\fi
\begin{thebibliography}{10}

\bibitem{Val}
V.~Alexeev, {\em Complete moduli in the presence of semiabelian group action},
  Annals of Mathematics {\bf 155} (2002), no.~3, 611--708.

\bibitem{Alexeevski2005}
A.~Alexeevski, {\em {Component Groups of the Centralizers of Unipotent Elements
  in Semisimple Algebraic Groups}}, Amer. Math. Soc. Transl. (2) {\bf 213}
  (2005), 15 -- 31.

\bibitem{BBD}
A.~Beilinson, J.~Bernstein, and P.~Deligne, {\em {Faisceaux pervers}},
  Ast\'{e}risque {\bf 100} (1982), 5--171.

\bibitem{Borho1983}
W.~Borho and R.~MacPherson, {\em {Partial resolutions of nilpotent varieties}},
  Analysis and topology on singular spaces, II, III (Luminy, 1981),
  Ast\'{e}risque {\bf 101-102} (1983), 23--74.

\bibitem{Brion1999}
M.~Brion, {\em {Rational smoothness and fixed points of torus actions}},
  Transformation Groups {\bf 4} (1999), no.~2-3, 127--156.

\bibitem{NilOrbBk}
D.~H. Collingwood and W.~M. McGovern, {\em {Nilpotent Orbits in Semisimple Lie
  Algebras}}, Van Nostrand Reinhold, New York, 1993.

\bibitem{Fulton}
W.~Fulton, {\em {Introduction to Toric Varieties}}, Princeton University Press,
  Princeton, New Jersey, 1993.

\bibitem{Graham}
W.~Graham, {\em Toric varieties and a generalization of the Springer
  resolution}, 2019, \mbox{1911.09601}.

\bibitem{HumphreysLieAlg}
J.~E. Humphreys, {\em {Introduction to Lie algebras and representation
  theory}}, Springer Verlag, New York, 1972.

\bibitem{Jantzen2004}
J.~C. Jantzen, Nilpotent Orbits in Representation Theory, 1--211,
  Birkh{\"a}user Boston, Boston, MA, 2004, pp.~1--211.

\bibitem{Kostant1963}
B.~Kostant, {\em {Lie Group Representations on Polynomial Rings}}, Amer. J. of
  Math. {\bf 85} (1963), no.~3, 327 -- 404.

\bibitem{Lu84}
G.~Lusztig, {\em {Intersection cohomology complexes on a reductive group}},
  Inv. Math. {\bf 75} (1984), no.~2, 205--272.

\bibitem{Mizuno80}
K.~Mizuno, {\em {The conjugate classes of unipotent elements of the Chevalley
  groups $E_7$ and $E_8$}}, Tokyo J. Math. {\bf 3} (1980), no.~2, 391 -- 459.

\bibitem{diss}
A.~Russell, {\em Graham's variety and perverse sheaves on the nilpotent cone},
  Ph.D. thesis, Louisiana State University,
  \url{https://digitalcommons.lsu.edu/gradschool_dissertations/3631}, 2012.

\bibitem{shim}
N.~Shimomura, {\em A theorem on the fixed point set of a unipotent
  transformation on the flag manifold}, J. Math. Soc. Japan {\bf 32} (1980),
  no.~1, 55--64.

\bibitem{Springer1976}
T.~Springer, {\em {Trigonometric sums, Green functions of finite groups and
  representations of Weyl groups}}, Inventiones Mathematicae {\bf 36} (1976),
  no.~1, 173--207.

\bibitem{Springer1978}
T.~Springer, {\em {A construction of representations of Weyl groups}},
  Inventiones Mathematicae {\bf 44} (1978), no.~3, 279--293.

\bibitem{TymLin}
J.~S. Tymoczko, {\em Linear Conditions Imposed on Flag Varieties}, American
  Journal of Mathematics {\bf 128} (2006), no.~6, 1587--1604.

\end{thebibliography}

\end{document}